\documentclass[12pt]{amsart}%
\usepackage{amsmath}
\usepackage{amsfonts}
\usepackage{amssymb}
\usepackage{graphicx}%
\providecommand{\U}[1]{\protect\rule{.1in}{.1in}}
\newtheorem{theorem}{Theorem}

\newtheorem{corollary}[theorem]{Corollary}

\newtheorem{lemma}[theorem]{Lemma}

\begin{document}

\title[Asymptotic Plateau's problem on rank 1 symmetric
spaces]{On the asymptotic Plateau's problem for CMC hypersurfaces on rank 1 symmetric
spaces of noncompact type}
\author[Jean-Baptiste Casteras]{Jean-Baptiste Casteras}
\address{UFRGS, Instituto de Matem\'atica, Av. Bento Goncalves 9500, 91540-000 Porto
Alegre-RS, Brasil.}
\email{jean-baptiste.casteras@univ-brest.fr}
\author[Jaime B. Ripoll]{Jaime B. Ripoll}
\email{jaime.ripoll@ufrgs.br}
\thanks{The first author was supported by the CNPq (Brazil) project 501559/2012-4. The second author was supported by the CNPq (Brazil) project 302955/2011-9.}
\date{}

\begin{abstract}
Let $M^{n},$ $n\geq3,$ be a Hadamard manifold with strictly negative sectional
curvature $K_{M}\leq-\alpha,$ $\alpha>0$. Assume that $M$ satisfies the
\emph{strict convexity condition }at infinity according to \cite{RTe} (see also the definition below) and, additionally, that $M$ admits a
\emph{helicoidal} one parameter subgroup $\left\{  \varphi_{t}\right\}  $ of
isometries (i.e. there exists a geodesic $\gamma$ of $M$ such that $\varphi
_{t}\left(  \gamma(s)\right)  =\gamma(t+s)$ for all $s,t\in\mathbb{R}$)$.$ We
then prove that, given a compact topological $\left\{  \varphi_{t}\right\}  -$shaped hypersurface $\Gamma$ in the
asymptotic boundary $\partial_{\infty}M$ of $M$ (that is,  the orbits of
the extended action of $\left\{  \varphi_{t}\right\}  $ to $\partial_{\infty
}M$ intersect $\Gamma$ at one and only one point), and given $H\in\mathbb{R}$, $\left\vert H\right\vert
<\sqrt{\alpha},$ there exists a complete properly embedded constant mean
curvature (CMC) $H$ hypersurface $S$ of $M$ such that $\partial_{\infty
}S=\Gamma.$

This result extends Theorem 1.8 of B. Guan and J. Spruck \cite{GS} to more general ambient spaces, as rank $1$ symmetric spaces of noncompact type, and allows  $\Gamma$ to be $\left\{  \varphi_{t}\right\}  -$shaped with respect to more general one parameter subgroup of isometries $\left\{  \varphi_{t}\right\}$ of the ambient space.  For example, in $\mathbb{H}^n$, $\Gamma$ can be \emph{loxodromic }$-$shaped, where loxodromic is a curve in  $\mathbb{S}^{n-1}=\partial_{\infty}\mathbb{H}^n$  that makes a constant angle with a family of circles connecting two points of $\mathbb{S}^{n-1}.$
A fundamental result used to prove our main theorem, which has
interest on its own, is the extension of the interior gradient estimates for
CMC Killing graphs proved in Theorem 1 of \cite{DLR} to CMC graphs of Killing submersions.

\end{abstract}

\maketitle

\section{Introduction}

\qquad Let $M^{n}$ be a Cartan-Hadamard manifold (namely a simply connected,
complete Riemannian manifold with nonpositive sectional curvature) of
dimension $n\geq3.$

The asymptotic boundary $\partial_{\infty}M$ of $M$ is defined as the set of
all equivalence classes of unit speed geodesic rays in $M$; two such rays
$\gamma_{1},\gamma_{2}:\left[  0,\infty\right)  \rightarrow M$ are equivalent
if $\sup_{t\geq0}d\left(  \gamma_{1}(t),\gamma_{2}(t)\right)  <\infty$, where
$d$ is the Riemannian distance in $M.$ The so called \emph{geometric}
compactification $\overline{M}$ of $M$ is then given by $\overline{M}%
:=M\cup\partial_{\infty}M,$ endowed with the cone topology (see \cite{EO} or \cite{SY}, Ch. 2). For any subset $S\subset M$, we define $\partial
_{\infty}S=\partial_{\infty}M\cap\overline{S}.$

The asymptotic Plateau problem for $k$ $(\geq2)$ dimensional area minimizing
submanifolds in $M$ consists in finding, for a given a $k-1$ dimensional, closed,
topological submanifold $\Gamma$ of $\partial_{\infty}M$, a locally area
minimizing$,$ complete submanifold $S^{k}$ of $M$ such that $\partial_{\infty
}S=\Gamma$.

By using methods from the Geometric Measure Theory, this problem was first studied
in the hyperbolic space by M.T. Anderson \cite{A} and his results extended to Gromov
hyperbolic manifolds by U. Lang and V. Bangert (\cite{BL}, \cite{La2},
\cite{La1}).

Within the framework of the classical Plateau problem, the second author of the
present paper with F. Tomi \cite{RT} study the asymptotic problem for
minimal disk type surfaces in a general Hadamard manifold $M$.

In codimension $1,$ given $H\in\mathbb{R}$, we may consider the asymptotic
Plateau's problem for the constant mean curvature (CMC)\ $H$ hypersurface in
$M,$ namely, given a compact topological hypersurface $\Gamma\subset
\partial_{\infty}M,$  find a complete CMC $H$ hypersurface $S$ of $M$
($H-$hypersurface, for short) such that $\partial_{\infty}S=\Gamma.$ This
problem has also attracted the attention of many mathematicians more recently.
The results of M.T. Anderson \cite{A} have been extended to the CMC case by Y.
Tonegawa \cite{T} and H. Alencar and H. Rosenberg in \cite{AR}.

Both Geometric Measure theory and Plateau's technique are methods that lead,
in general, to the existence of hypersurfaces with singularities. Thus, a natural question,
raised by B. Guan and J. Spruck in \cite{GS}, asks about the existence of a smooth
constant mean curvature hypersurface asymptotic to $\Gamma$ at infinity in
$\mathbb{H}^{n}$.   This problem in fact had already been studied earlier in the
minimal case by F. H. Lin \cite{L}.

  A way to obtain smooth solutions is by finding a suitable system of coordinates in order to write the hypersurface as a graph, and then to use standard elliptic PDE methods. In \cite{L}, F.H. Lin represented the hypersurfaces in the half space model of $\mathbb{H}^{n}$ as vertical graphs, that is, in the usual way of $\mathbb{R}^{n}_{+}$ when using the cartesian system of coordinates.

The results of F.H. Lin \cite{L} were extended to the CMC case by B. Nelli
and J. Spruck in \cite{NS} where they proved the existence of a smooth CMC
$\left\vert H\right\vert <1$ hypersurface in the hyperbolic space
$\mathbb{H}^{n}$ with sectional curvature $-1$ if $\Gamma$ is assumed to be
convex and compact. Later, also using PDE's techniques, B. Guan and J. Spruck
\cite{GS} (see also \cite{SS} for a different approach based on a variational
method) improved the convexity condition by requiring a starshaped property of
$\Gamma$.  We refer the reader to
the nice survey of B. Coskunuzer \cite{Cos}, where the references of many
other closely related papers to this subject can be found.

In both papers \cite{NS} and \cite{GS} the authors used the underlying
Euclidean structure of the half space model for $\mathbb{H}^{n}$ to state the
convexity and starshaped properties of $\Gamma.$ However, although the
convexity is not an intrinsic notion of the hyperbolic geometry, the starshapness
of $\Gamma$ is. It can be formulated in intrinsic terms using the conformal
structure of $\overline{\mathbb{H}}^{n}$ by requiring $\Gamma$ to be
\textquotedblleft circle shaped\textquotedblright, meaning that there are two
points $p_{1},p_{2}\in\mathbb{S}^{n-1}=\partial_{\infty}\mathbb{H}^{n}$ such
that any arc of circle from $p_{1}$ to $p_{2}$ intersects $\Gamma$ at one and
only one point. A limit  circle shaped condition, where $p_{1}=p_{2},$ was also
introduced and used by the second author in \cite{R} to ensure the existence
of a smooth CMC hypersurface having $\Gamma$ as asymptotic boundary (see the
Introduction and Theorem 6 of \cite{R} for a detailed description of this case).

In the present paper we extend Theorem 1.8 of \cite{GS} in two directions.
First, we allow $\Gamma$ to be \textquotedblleft shaped\textquotedblright%
\ with respect to a more general one parameter subgroup of conformal
diffeomorphisms of $\mathbb{S}^{n-1}=\mathbb{\partial}_{\infty}\mathbb{H}%
^{n}.$ Secondly, we allow the ambient space to be any rank 1 symmetric space
of noncompact type. 
Both results are consequences of a more general theorem
that holds in a Hadamard manifold endowed with some special Killing field. 

As we shall see in the proof ahead, the Killing field allows to introduce a special system of coordinates which is quite suitable for using standard elliptic PDE techniques. To
write down precise statements we first introduce some general notions and terminology.

Let $\gamma:\left(  -\infty,\infty\right)  \rightarrow M$ be an arc length
geodesic. We say that a one parameter subgroup of isometries $\left\{
\varphi_{t}^{\gamma}\right\}  $ of $M$ associated to $\gamma$ is
\emph{helicoidal} if
$\varphi_{t}^{\gamma}\left(  \gamma\left(  s\right)  \right)  =\gamma\left(
t+s)\right)  $ for all $s,t\in\mathbb{R}$.
In the sequel, since there is no possibility of confusion, we shall omit the
dependance of $\left\{  \varphi_{t}^{\gamma}\right\}  $ with respect to the
geodesic $\gamma$.

 Let us illustrate the previous definition with a simple case that justifies this terminology:
If $M=\mathbb{R}^{3}$ then any helicoidal one parameter subgroup of
isometries, up to a conjugation, is of the form%
\[
\varphi_{t}\left(  x,y,z\right)  =\left(  \left[
\begin{array}
[c]{cc}%
\cos at & \sin at\\
-\sin at & \cos at
\end{array}
\right]  \left[
\begin{array}
[c]{c}%
x\\
y
\end{array}
\right]  ,z+t\right)
\]
for some $a\in\mathbb{R}$. When $a=0$, $\left\{  \varphi_{t}\right\}  $ is a
one parameter subgroup of transvections along the $z-$axis. More generally, a
one parameter subgroup of transvections along a geodesic in a symmetric space
(see \cite{He}) is a particular case of helicoidal one parameter subgroup of
isometries.

Since the equivalence relation between geodesics and convergent sequences are
preserved under isometries, the action of $\left\{  \varphi_{t}\right\}  $ on
$M$ extends to the compactification $\overline{M}$ of $M$ and the extended
action is continuous. The orbits of $\left\{  \varphi_{t}\right\}  $ are the
curves $O(x):=\left\{  \varphi_{t}(x)\text{
$\vert$
}t\in\mathbb{R}\right\}  $ where $x\in\overline{M}.$ Observe that $\left\{
\varphi_{t}\right\}  $ has two singular orbits in $\overline{M},$ namely,
$O\left(  \gamma\left(  \pm\infty\right)  \right)  ,$ where $\gamma$ is the
geodesic translated by $\left\{  \varphi_{t}\right\}  .$

Finally, we will also need to use the \emph{Strictly Convexity Condition
}(\textquotedblleft SC condition\textquotedblright) introduced in \cite{RT}.
We say that $M$ satisfies the SC condition if, given $x\in\partial_{\infty}M$
and a relatively open subset $W\subset\partial_{\infty}M$ containing $x$,
there exists a $C^{2}$ open set $\Omega\subset\overline{M}$ such that
$x\in\operatorname*{Int}(\partial_{\infty}\Omega)\subset W$ and $M\backslash
\Omega$ is convex, where $\operatorname*{Int}(\partial_{\infty}\Omega)$ stands
for the interior of $\partial_{\infty}\Omega$ in $\partial_{\infty}M$.

We are now in position to state our main result :

\begin{theorem}
\label{main}Let $M$ be a Hadamard manifold with sectional curvature $K_{M}%
\leq-\alpha$, for some $\alpha>0$, satisfying the SC condition$.$ Let
$\left\{  \varphi_{t}\right\}  $ be a helicoidal one parameter subgroup of
isometries of $M$. Let $\Gamma\subset\partial_{\infty}M$ be a compact
topological embedded $\left\{  \varphi_{t}\right\}  -$shaped hypersurface of $\partial_{\infty}M,$ that is, any
nonsingular orbit of $\left\{  \varphi_{t}\right\}  $ in $\partial_{\infty}M$
intersects $\Gamma$ at one and only one point. Then, given $H\in\mathbb{R}$,
$\left\vert H\right\vert <\sqrt{\alpha},$ there exists a
complete, properly embedded $H-$hypersurface $S$ of $M$ such that
$\partial_{\infty}S=\Gamma.$ Moreover any orbit of $\left\{  \varphi
_{t}\right\}  $ intersects $S$ at one and only one point.
\end{theorem}

We point out that the SC condition is satisfied by a large class of manifolds.
For example, if the sectional curvature is bounded from above by a strictly
negative constant and decreases at most exponentially (see Theorem $14$ of
\cite{RTe}) then the SC condition is satisfied. In particular, it is satisfied by any rank 1 symmetric spaces of
noncompact type. Therefore, as an immediate consequence of the previous
theorem, we obtain:

\begin{corollary}
\label{cor1}Assume that $M$ is a rank 1 symmetric space of noncompact type and
assume that the sectional curvature of $M$ is bounded by $-\alpha,$
$\alpha>0.$ Let $\left\{  \varphi_{t}\right\}  $ be a one parameter of
transvections of $M$. Let $\Gamma\subset\partial_{\infty}M$ be a compact
embedded topological $\left\{  \varphi_{t}\right\}  -$shaped hypersurface of $\partial_{\infty}M.$  Then, given $H\in\mathbb{R}$, $\left\vert H\right\vert <\sqrt{\alpha},$ there exists a complete, properly embedded
$H-$hypersurface $S$ of $M$ such that $\partial_{\infty}S=\Gamma.$ Moreover
any orbit of $\left\{  \varphi_{t}\right\}  $ intersects $S$ at one and only
one point.
\end{corollary}

Finally we point out an interesting corollary of Theorem \ref{main} in the
case where $M=\mathbb{H}^{n}$, the hyperbolic space of constant sectional
curvature $-1$. Recalling that a \emph{loxodromic} \emph{curve} is a curve in
$\mathbb{S}^{n-1}$ that intersects with a constant angle $\theta$ any arc of circle of
$\mathbb{S}^{n-1}$ connecting two fixed points of $\mathbb{S}^{n-1}$ (see
\cite{Lox}). These curves are induced by one-parameter subgroups of isometries of  $\mathbb{H}^{n}$ of helidoidal type. For example, in the half space model ${z>0}$ of $\mathbb{H}^{3}$, up to conjugation, they are of the form
\[
\varphi_{t}\left(  x,y,z\right)  =e^{t}\left(  \left[
\begin{array}
[c]{cc}%
\cos \theta t   & \sin  \theta  t \\
-\sin  \theta t   & \cos  \theta t
\end{array}
\right]  \left[
\begin{array}
[c]{c}%
x\\
z
\end{array}
\right]  ,z\right).
\]

\begin{corollary}
\label{cor2}
Let $0\leq\theta<\pi/2$ and $p_{1},p_{2}\in\mathbb{S}^{n-1}=\partial_{\infty
}\mathbb{H}^{n}$ be two distinct points of $\mathbb{S}^{n-1}.$ Let $L_{\theta
}$ be the family of loxodromic curves that intersects any arc of circle from
$p_{1}$ to $p_{2}$ with a constant angle $\theta.$ Let $\Gamma\subset
\mathbb{S}^{n-1}$ be a compact embedded topological $L_{\theta
}-$shaped hypersurface of
$\mathbb{S}^{n-1}.$ Then, given $H\in\mathbb{R}$, $\left\vert H\right\vert <1,$ there exists a complete, properly embedded $H-$hypersurface $S$ of
$\mathbb{H}^{n}$ such that $\partial_{\infty}S=\Gamma.$
\end{corollary}

We notice that, taking $\theta=0$ in the previous corollary, we recover
Theorem 1.8 of \cite{GS}. Theorem 1.8 also follows from Corollary \ref{cor1}
since radial graphs (considered in \cite{GS}) are transvections along a
geodesic of $\mathbb{H}^{n}.$

A fundamental result for proving the above theorems, which has interest on its
own, are the interior gradient estimates of the solutions of the CMC $H$ graph
PDE for Killing submersions (see Theorem \ref{intgrad} below). It extends
Theorem 1 of \cite{DLR}.

\section{Proofs of the results}

\qquad In what follows we use most of the nomenclature and the results proved by M. Dajczer and J. H. de Lira in
\cite{DL}. However, we introduce the notion of a Killing graph on a manner slightly
different from the one considered in \cite{DL}.

For the next result we allow $M$ be any Riemannian manifold and $Y$ a Killing field in $M$ without
singularities$.$ Denote by $\mathcal{O}\left(  x\right)  $ the integral curve
(which we also call orbit) of $Y$ through a point $x\in M.$ By a complete $Y-$
Killing section (we shall refer only to a Killing section because
$Y$ will be fixed throughout the text) we mean a complete up to the boundary
(possibly empty) hypersurface $P$ of $M$ such that any orbit $\mathcal{O}%
\left(  p\right)  $ of $Y$ through a point $p$ of $P$ intersects $P$ only at
$p$ and the intersection is transversal. We call $\Omega:=P\backslash\partial P$ a Killing domain. If  $P=\overline{\Omega}$ is a
hypersurface of class $C^{2,\alpha}$ in $M$ then we say that $\Omega$ is a $C^{2,\alpha}$ Killing domain.

If $u$ is a
function defined on a subset $T$ of $P$, the Killing graph of $u$ is given by
\[
\operatorname*{Gr}\left(  u\right)  =\left\{  \varphi\left(  u\left(  p\right)
,p\right)  \text{
$\vert$
}p\in T\right\}
\]
where $\varphi(s,x)=\varphi_{s}(x)$ is the flow of $Y.$ In the sequel, $s$ will stand for the flow parameter. We also set 
$$\Gamma_{T}=\left\{
\varphi(s,x)\text{
$\vert$
}x\in T\text{ and }s\in\mathbb{R}\right\}  .$$

Next, we denote by $\Pi:M\rightarrow P$ the projection defined by $\Pi(x)=\mathcal{O}%
\left(  x\right)  \cap P$. In all the sequel, we endow $P$ with the Riemannian metric $\left\langle \text{
, }\right\rangle _{\Pi}$ such that $\Pi$ becomes a Riemannian
submersion. 

Assume that $\Omega$ is a $C^{2,\alpha}$ Killing domain. 
Given $H\in\mathbb{R}$, it is not difficult to show that $\operatorname*{Gr}\left(  u\right)  $ has
CMC $H$ with respect to the unit normal vector field $\eta$ to
$\operatorname*{Gr}\left(  u\right)  $ such that $\left\langle Y,\eta
\right\rangle \leq0$ if and only if $u$ satisfies a certain second order
quasi-linear elliptic PDE $Q_{H}\left[  u\right]  =0$ on $M$ in terms of the
metric $\left\langle \text{ , }\right\rangle _{\Pi}$ in $P$ (for details,
including an explicit expression of $Q_H,$ see Section 2.1 of \cite{DL} or the short
revision done below).

We may then refer to the CMC\ $H$ Dirichlet problem in a Killing domain
$\Omega\subset M$ and for a given boundary data $\phi\in C^{0}\left(
\partial\Omega\right)  $  as the PDE boundary problem%
\begin{equation}
\left\{
\begin{array}
[c]{ll}%
Q_H\left[  u\right]  =0\text{ in }\Omega & u\in C^{2,\alpha}\left(
\Omega\right)  \cap C^{0}\left(  \overline{\Omega}\right) \\
u|_{\partial\Omega}=\phi. &
\end{array}
\right.  \label{eqini}%
\end{equation}

We begin by obtaining interior gradient estimates for the solutions of \eqref{eqini}. Our result generalizes Theorem 1 of \cite{DLR} to the case of CMC $H$ graph
PDE of Killing submersions. 

Fix a point $o\in \Omega$ and let $r>0$ be such that $r<i(o)$, the injectivity radius of $M$ at $o$. We obtain the following result :
\begin{theorem}
\label{intgrad}Let $\Omega$ be a Killing domain in $M.$ Let $o\in\Omega$ and
$r>0$ such that the open geodesic ball $B_{r}\left(  o\right)  $ is contained
in $\Omega.$ Let $u\in C^{3}\left(  B_{r}\left(  o\right)  \right)  $ be a
negative solution of $Q[u]=0$ in $B_{r}\left(  o\right)  .$ Then there is a
constant $L$ depending only on $u(o),$ $r,$ $\left\vert Y\right\vert $ and $H$
such that $\left\vert \nabla u\left(  o\right)  \right\vert \leq L.$
\end{theorem}

Before proving the above theorem, we review the nomenclature and some facts
of \cite{DL}. 

We fix a local reference frame $v_1,\ldots ,v_n$ on $\overline{\Omega}$ and we set $\sigma_{ij}=\left\langle v_i, v_j\right\rangle_{\Pi}$. We will now define a local frame in $M$. We denote by $D_1,\ldots ,D_n$ the basic vector fields $\Pi$-related to $v_1,\ldots ,v_n$.
The frame $D_0,\ldots ,D_n$, we considered on $M$, is defined by $D_0=f^{\frac{1}{2}}\partial_s$, where $f=\dfrac{1}{|Y|^2}$, ($\partial_s (q)=\varphi_\ast (s,p)\partial_s(p)$), and $D_i(q)=\varphi_\ast (s,p) D_i(p)$, where $q=\varphi (s,p)$, $p\in P$. 
 We point out that the unit normal vector field to $Gr(u)$ pointing upward is given by
\begin{equation}
\label{normal}
N=\dfrac{1}{W}(f^{\frac{1}{2}}D_0-\hat{u}^j D_j),
\end{equation}
where $\hat{u}^j=\sigma^{ij} D_i (u-s)$ and $W^2=f+\hat{u}^i \hat{u}_i=f+ \sigma_{ij} \hat{u}^i \hat{u}^j$. We notice that $ \hat{u}_i$ and $W$ are not depending on $s$ and therefore can be seen as function defined on $P$. Finally, using the previous notation, the operator $Q$ (defined in \eqref{eqini}) can be written as
$$Q_H\left[u\right]=\dfrac{1}{W}(A^{ij}\hat{u}_{j;i}-\dfrac{(f+W^2)}{W^2}\left\langle \Pi_\ast \bar{\nabla}_{D_0}D_0,Du \right\rangle )-nH,$$
where $\hat{u}_{i;j}=\left\langle \bar{\nabla}_{D_i} \bar{\nabla}(u-s),D_j  \right\rangle$, $Du=\Pi_\ast \overline{\nabla }(u-s)$ and $A^{ik}= \sigma^{ik}-\dfrac{\hat{u}^i \hat{u}^k}{W^2}$.

\bigskip

\begin{proof}
[Proof of Theorem \ref{intgrad}] The proof will follow closely the one of Theorem $1$ in \cite{DLR}. Let $p\in B_{r}(o)$ be an interior point where
$h=\eta W$ attains its maximum, where $\eta$ is a smooth function with support
in $B_{r}(o)$ which will be determined in the sequel. In all this section, the
computations will be done at the point $p$. Let $v_{1},\ldots,v_{n}$ be an
orthonormal tangent frame at $p\in B_{r}(o)$. Then we have $h_{i}=0$ (where
the derivative is taken with respect to $v_{i}$). This implies that
\begin{equation}
\eta_{i}W=-\eta W_{i}. \label{critpoint}%
\end{equation}
We also have, since $\dfrac{A^{ij}}{W}$ is definite positive, that
\[
0\geq\dfrac{1}{W}A^{ij}h_{ij}=\dfrac{1}{W}A^{ij}(W\eta_{i;j}+2\eta_{i}%
W_{j}+\eta W_{i;j}).
\]
Using \eqref{critpoint}, the previous inequality can be rewritten as
\begin{equation}
A^{ij}\eta_{i;j}+\dfrac{\eta}{W^{2}}A^{ij}(WW_{i;j}-2W_{i}W_{j})\leq0.
\label{critpoint1}%
\end{equation}
From \eqref{normal}, we have
\begin{equation}
N^{k}=-\dfrac{\hat{u}^{k}}{W}. \label{normale1}%
\end{equation}
Derivating $W$, we find
\[
W_{i}=\dfrac{f_{i}}{2W}+\dfrac{\hat{u}^{k}\hat{u}_{k;i}}{W}=\dfrac{f_{i}}%
{2W}-N^{k}\hat{u}_{k;i}.
\]
From \eqref{normale1}, we get
\[
N_{;j}^{k}=-\dfrac{\hat{u}_{;j}^{k}}{W}+\dfrac{\hat{u}^{k}W_{j}}{W^{2}}.
\]
Using the previous inequalities, we have
\begin{align*}
W_{i;j}  &  =\dfrac{f_{i;j}}{2W}-\dfrac{f_{i}W_{j}}{2W^{2}}-N_{;j}^{k}\hat
{u}_{k;i}-N^{k}\hat{u}_{k;ij}\\
&  =\dfrac{f_{i;j}}{2W}+\dfrac{\hat{u}_{;j}^{k}\hat{u}_{k;i}}{W}-\dfrac
{\hat{u}^{k}\hat{u}_{k;i}W_{j}}{W^{2}}-\dfrac{f_{i}W_{j}}{2W^{2}}-N^{k}\hat
{u}_{k;ij}\\
&  =\dfrac{f_{i;j}}{2W}+\dfrac{\hat{u}_{;j}^{k}\hat{u}_{k;i}}{W}-\dfrac
{W_{i}W_{j}}{W}-N^{k}\hat{u}_{k;ij}\\
&  =\dfrac{f_{i;j}}{2W}+\dfrac{A^{kl}}{W}\hat{u}_{l;j}\hat{u}_{k;i}%
+\dfrac{f_{i}f_{j}}{4W^{3}}-\dfrac{1}{2W^{2}}(W_{i}f_{j}+W_{j}f_{i})-N^{k}%
\hat{u}_{k;ij}.
\end{align*}
Multiplying by $A^{ij}$ the above equation and using \eqref{critpoint}, we
find
\begin{equation}
A^{ij}W_{i;j}=\dfrac{A^{ij}f_{i;j}}{2W}+\dfrac{A^{ij}A^{kl}}{W}\hat{u}%
_{l;j}\hat{u}_{k;i}+\dfrac{A^{ij}f_{i}f_{j}}{4W^{3}}+\dfrac{1}{\eta W}%
A^{ij}\eta_{i}f_{j}-A^{ij}N^{k}\hat{u}_{k;ij}. \label{eqint1}%
\end{equation}
In order to get rid of the term involving three derivatives of $u$ in
\eqref{eqint1}, we want to find a commutation formula for $\hat{u}_{k;ij}$. We
recall (see equation ($11$) of \cite{DL}) that
\[
\hat{u}_{k;i}=u_{k;i}-s_{k;i}+\dfrac{1}{2}\gamma_{ki},
\]
where $\gamma_{ki}=f^{\frac{1}{2}}\left\langle [D_{k},D_{i}],D_{0}%
\right\rangle $. We deduce from the previous equality that
\begin{align*}
\hat{u}_{k;ij}  &  =u_{k;ij}-s_{k;ij}+\dfrac{1}{2}(\gamma_{ki})_{j}%
=u_{i;jk}+R_{kji}^{l}u_{l}-s_{k;ij}+\dfrac{1}{2}(\gamma_{ki})_{j}\\
&  =(\hat{u}_{j;i}+s_{j;i}-\dfrac{1}{2}\gamma_{ji})_{k}+R_{kji}^{l}%
u_{l}-s_{k;ij}+\dfrac{1}{2}(\gamma_{ki})_{j}\\
&  =\hat{u}_{j;ik}+s_{j;ik}-\dfrac{1}{2}(\gamma_{ji})_{k}+R_{kji}^{l}%
u_{l}-s_{k;ij}+\dfrac{1}{2}(\gamma_{ki})_{j}\\
&  =\hat{u}_{j;ik}+R_{kji}^{l}\hat{u}_{l}+R_{kji}^{l}s_{l}+s_{j;ik}%
-s_{k;ij}+\dfrac{1}{2}((\gamma_{ki})_{j}-(\gamma_{ji})_{k})\\
&  =\hat{u}_{j;ik}+R_{kji}^{l}\hat{u}_{l}+C_{ijk},
\end{align*}
where $C_{ijk}=R_{kji}^{l}s_{l}+s_{j;ik}-s_{k;ij}+\dfrac{1}{2}((\gamma
_{ki})_{j}-(\gamma_{ji})_{k})$ is not depending on $u$.
Using \eqref{eqini} and the commutation formula, the last term of
\eqref{eqint1} rewrites as
\begin{align*}
A^{ij}N^{k}\hat{u}_{k;ij}  &  =A^{ij}N^{k}\hat{u}_{j;ik}-\dfrac{A^{ij}%
R_{kji}^{l}\hat{u}_{l}\hat{u}^{k}}{W}-\dfrac{\hat{u}^{k}A^{ij}C_{ijk}}{W}\\
&  =N^{k}(A^{ij}\hat{u}_{j;i})_{k}-N^{k}A_{;k}^{ij}\hat{u}_{i;j}-\dfrac
{A^{ij}R_{kji}^{l}\hat{u}_{l}\hat{u}^{k}}{W}-\dfrac{\hat{u}^{k}A^{ij}C_{ijk}%
}{W}\\
&  =nN^{k}(WH)_{k}+N^{k}\left(  \dfrac{(f+W^{2})}{W^{2}}\left\langle \Pi
_{\ast}\bar{\nabla}_{D_{0}}D_{0},Du\right\rangle \right)  _{k}\\
&  -N^{k}A_{;k}^{ij}\hat{u}_{j;i}-\dfrac{A^{ij}R_{kji}^{l}\hat{u}_{l}\hat
{u}^{k}}{W}-\dfrac{\hat{u}^{k}A^{ij}C_{ijk}}{W}.
\end{align*}
Straightforward computations using \eqref{critpoint} give
\[
(WH)_{k}=W_{k}H+WH_{k}=\dfrac{W}{\eta}(-\eta_{k}H+\eta H_{k}),
\]
and
\begin{equation}
(\dfrac{f+W^{2}}{W^{2}})_{k}=\dfrac{f_{k}}{W^{2}}-\dfrac{f}{W^{4}}(f_{k}%
+2\hat{u}^{l}\hat{u}_{l;k})=\dfrac{1}{W^{2}}(f_{k}+2\dfrac{f\eta_{k}}{\eta}).
\label{612e1}%
\end{equation}
We also have, using \eqref{612e1},
%

\begin{align*}
\left(  \dfrac{(f+W^{2})}{W^{2}}\left\langle \Pi_{\ast}\bar{\nabla}_{D_{0}%
}D_{0},Du \right\rangle \right)  _{k} 
 &  = \dfrac{(f+W^{2})}{2fW^{2}}[(\dfrac{f_{l} f_{k}}{f}-f_{k;l})WN^{l}+ f^{l} \hat{u}_{l;k}] \\
 & + \left\langle \Pi_{\ast}\bar{\nabla}_{D_{0}}D_{0},Du \right\rangle
\dfrac{1}{W^{2}} (f_{k} +2 \dfrac{f\eta_{k}}{ \eta}),
\end{align*}
and
\begin{align*}
A^{ij}_{;k}  &  =-\dfrac{1}{W^{2}}(\hat{u}^{i}_{;k} \hat{u}^{j}+\hat{u}^{i}
\hat{u}^{j}_{;k})+\dfrac{1}{W^{4}} (f_{k}- 2WN^{l} \hat{u}_{l;k})\hat{u}^{i}
\hat{u}^{j}\\
&  = \dfrac{1}{W} (\hat{u}^{i}_{;k}-N^{i} N^{l} \hat{u}_{l;k} )N^{j}
+\dfrac{1}{W} (\hat{u}^{j}_{;k}-N^{i} N^{l} \hat{u}_{l;k} )N^{i} +\dfrac
{1}{W^{2}} f_{k} N^{i} N^{j}\\
&  = \dfrac{1}{W} A^{il} \hat{u}_{l;k} N^{j}+ \dfrac{1}{W} A^{jl} \hat
{u}_{l;k} N^{i}+\dfrac{1}{W^{2}} f_{k} N^{i} N^{j}.
\end{align*}
Multiplying the previous equality by $N^{k} \hat{u}_{j;i}$, we find
\[
N^{k} A^{ij}_{;k}\hat{u}_{j;i}= \dfrac{1}{W}N^{k} \hat{u}_{j;i}\hat{u}_{l;k}
(A^{il}N^{j}+A^{jl}N^{i} )+\dfrac{1}{W^{2}} f_{k} N^{i} N^{j} N^{k} \hat
{u}_{j;i} .
\]
Recalling that
\[
N^{k} \hat{u}_{k;i}=\dfrac{f_{i}}{2W}+\dfrac{W \eta_{i}}{\eta},
\]
and
\[
\hat{u}_{i;j}=\hat{u}_{j;i}+\gamma_{ij},
\]
we have
\begin{align*}
N^{k} A^{ij}_{;k}\hat{u}_{j;i}  &  =\dfrac{1}{W^{2}} f_{k} N^{k} N^{i}
(\dfrac{f_{i}}{2W}+\dfrac{W \eta_{i}}{\eta})+ \dfrac{1}{W} A^{il}(\dfrac
{f_{i}}{2W}+\dfrac{W \eta_{i}}{\eta})N^{k} (\hat{u}_{k;l}+\gamma_{lk} )\\
&  \hspace{3cm}+\dfrac{1}{W} A^{jl} N^{k} N^{i} (\hat{u}_{i;j}+\gamma
_{ji})(\hat{u}_{k;l}+\gamma_{lk})\\
&  = \dfrac{1}{W^{2}} f_{k} N^{k} N^{i} (\dfrac{f_{i}}{2W}+\dfrac{W \eta_{i}%
}{\eta})+ \dfrac{2}{W} A^{il}(\dfrac{f_{i}}{2W}+\dfrac{W \eta_{i}}{\eta
})(\dfrac{f_{l}}{2W}+\dfrac{W \eta_{l}}{\eta})\\
&  \hspace{2cm} +\dfrac{1}{W} A^{jl} N^{k} N^{i} \gamma_{ji} \gamma_{lk}+
\dfrac{3}{W} A^{il}(\dfrac{f_{i}}{2W}+\dfrac{W \eta_{i}}{\eta})N^{k}
\gamma_{lk},
\end{align*}
and
\[
N^{k} f_{l} \hat{u}_{l;k}= N^{k} f_{l} (\hat{u}_{k;l}-\gamma_{kl})= \dfrac
{f}{2W}\sigma^{kl}\dfrac{f_{k} f_{l}}{f}+\dfrac{W}{\eta}f_{l} \eta^{l}%
-\gamma_{kl}N^{k} f_{l} .
\]
Using the previous computations, we deduce that the last term of
\eqref{eqint1} can be rewritten as
\begin{align*}
A^{ij}  &  N^{k} \hat{u}_{k;ij}= n N^{k} \dfrac{W}{\eta} (-\eta_{k} H+\eta
H_{k})- \dfrac{2}{W} A^{il}(\dfrac{f_{i}}{2W}+\dfrac{W \eta_{i}}{\eta}%
)(\dfrac{f_{l}}{2W}+\dfrac{W \eta_{l}}{\eta})\\
&  \hspace{3cm}- \dfrac{1}{W^{2}} f_{k} N^{k} N^{i} (\dfrac{f_{i}}{2W}%
+\dfrac{W \eta_{i}}{\eta})- \dfrac{1}{W} A^{jl} N^{k} N^{i} \gamma_{ji}
\gamma_{lk}\\
&  \hspace{2cm}- \dfrac{3}{W} A^{il}(\dfrac{f_{i}}{2W}+\dfrac{W \eta_{i}}%
{\eta})N^{k} \gamma_{lk}-\dfrac{A^{ij}R_{kji}^{l} \hat{u}_{l} \hat{u}^{k}}%
{W}\\
&  \hspace{3cm}-\dfrac{\hat{u}^{k} A^{ij}C_{ijk}}{W}+ \left\langle \Pi_{\ast
}\bar{\nabla}_{D_{0}}D_{0},Du \right\rangle \dfrac{N_{k}}{W^{2}} (f_{k} +2
\dfrac{f\eta_{k}}{ \eta})\\
&  + \dfrac{(f+W^{2})}{2fW^{2}} \left[  \left(  \dfrac{f}{2W}%
\sigma^{kl}+WN^{k} N^{l}\right)  \dfrac{f_{k} f_{l}}{f}-W N^{k} N^{l}
f_{k;l}+\dfrac{W}{\eta}f_{l} \eta^{l}-\gamma_{kl}N^{k} f_{l}\right]  .
\end{align*}
Thus, from \eqref{eqint1}, we obtain
\begin{align*}
A^{ij}  &  W_{ij}-\dfrac{2}{W}A^{ij}W_{i} W_{j}\\
&  = \dfrac{3}{4W^{3}}A^{ij}f_{i} f_{j} +\dfrac{1}{W}A^{ij}A^{kl}\hat{u}%
_{l;j}\hat{u}_{k;i}+\dfrac{3}{W\eta}A^{ij}f_{i} \eta_{j}+\dfrac{1}{2W}%
A^{ij}f_{i;j}\\
&  - n N^{k} \dfrac{W}{\eta} (-\eta_{k} H+\eta H_{k}) + \dfrac{1}{W^{2}} f_{k}
N^{k} N^{i} (\dfrac{f_{i}}{2W}+\dfrac{W \eta_{i}}{\eta})+ \dfrac{1}{W} A^{jl}
N^{k} N^{i} \gamma_{ji} \gamma_{lk}\\
&  \hspace{3cm}+ \dfrac{3}{W} A^{il}(\dfrac{f_{i}}{2W}+\dfrac{W \eta_{i}}%
{\eta})N^{k} \gamma_{lk}+\dfrac{A^{ij}R_{kji}^{l} \hat{u}_{l} \hat{u}^{k}}%
{W}\\
&  \hspace{2.5cm}+\dfrac{\hat{u}^{k} A^{ij}C_{ijk}}{W}- \left\langle \Pi
_{\ast}\bar{\nabla}_{D_{0}}D_{0},Du \right\rangle \dfrac{N_{k}}{W^{2}} (f_{k}
+2 \dfrac{f\eta_{k}}{ \eta})\\
&  - \dfrac{(f+W^{2})}{W^{2}}\dfrac{1}{2f} \left[  \left(  \dfrac{f}{2W}%
\sigma^{kl}+WN^{k} N^{l}\right)  \dfrac{f_{k} f_{l}}{f}-W N^{k} N^{l}
f_{k;l}+\dfrac{W}{\eta}f_{l} \eta^{l}-\gamma_{kl}N^{k} f_{l}\right]  .
\end{align*}
Multiplying by $\dfrac{\eta}{W}$, we have
\begin{align*}
&  \dfrac{\eta}{W} (A^{ij}W_{ij}-\dfrac{2}{W}A^{ij}W_{i} W_{j} )\\
&  \hspace{1cm} \geq\left[  -nN^{k} H_{k} -\dfrac{f_{k}}{ W^{3}}N^{k}
\left\langle \Pi_{\ast}\bar{\nabla}_{D_{0}}D_{0},Du \right\rangle +\dfrac
{1}{2W^{2}}A^{ij}f_{i;j}\right. \\
&  +\left.  \dfrac{1}{W^{2}} A^{jl} N^{k} N^{i} \gamma_{ji} \gamma_{lk}%
+\dfrac{3}{2W^{3}} A^{il}f_{i} N^{k} \gamma_{lk} +\dfrac{\hat{u}^{k}
A^{ij}C_{ijk}}{W^{2}}+\dfrac{A^{ij}R_{kji}^{l} \hat{u}_{l} \hat{u}^{k}}{W^{2}%
}\right. \\
&  - \left.  \dfrac{(f+W^{2})}{W^{2}}\dfrac{1}{2f} \left[  \left(  \dfrac
{f}{2W^{2}}\sigma^{kl}+N^{k} N^{l}\right)  \dfrac{f_{k} f_{l}}{f}- N^{k} N^{l}
f_{k;l}-\dfrac{1}{W}\gamma_{kl}N^{k} f_{l} \right]  \right]  \eta\\
&  + \left[  \left(  nH+\dfrac{1}{W^{2}}N^{k} f_{k} -\dfrac{2f}{W^{3}%
}\left\langle \Pi_{\ast}\bar{\nabla}_{D_{0}}D_{0},Du \right\rangle \right)
N^{i} \right. \\
&  \hspace{2cm}+\left.  \dfrac{3}{W}A^{jl} N^{k} \gamma_{lk}+\left(  \dfrac
{3}{W^{2}}A^{ij}- \dfrac{(f+W^{2})}{W^{2}}\dfrac{1}{2f} \sigma^{ij} \right)
f_{j} \right]  \eta_{i} .
\end{align*}
Thus it is easy to see that there exists a constant $M>0$, not depending on
$u$, such that
\[
\dfrac{\eta}{W^{2}} (WA^{ij}W_{ij}-2 A^{ij}W_{i} W_{j} )\geq-M \eta-A^{i}
\eta_{i},
\]
where $A^{i}$ is the coefficient of $\eta_{i}$. From \eqref{critpoint1}, we
deduce that
\begin{equation}
\label{eqint2}A^{ij} \eta_{i;j} -M\eta-A^{i} \eta_{i} \leq0.
\end{equation}
We are now ready to choose an explicit $\eta$. We take
\[
\eta(x)=g(\phi(x))=e^{\displaystyle C_{1} \phi(x)}-1= e^{\displaystyle C_{1}
(1-\frac{d^{2}(x)}{r^{2}}+\frac{u(x)}{C})^{+}}-1,
\]
where $C=-\dfrac{1}{2 u(o)}.$ Straightforward computations give
\[
\eta_{i}=g^{\prime}(-r^{-2}(d^{2})_{i}+C(u_{i}-s_{i}) )=g^{\prime}%
(-r^{-2}(d^{2})_{i}+C\hat{u}_{i}),
\]
and
\[
\eta_{i;j}=g^{\prime}(-r^{-2}(d^{2})_{i;j}+C\hat{u}_{i;j} )+g^{\prime\prime}
(-r^{-2}(d^{2})_{i}+C\hat{u}_{i}) (-r^{-2}(d^{2})_{j}+C\hat{u}_{j}) .
\]
We deduce from the two previous lines that
\[
A^{ij}(-r^{-2}(d^{2})_{i}+C\hat{u}_{i}) (-r^{-2}(d^{2})_{j}+C\hat{u}_{j}%
)\geq\dfrac{C^{2} f}{W^{2}}\left(  |Du|^{2}-\dfrac{2}{Cr^{2}}\left\langle
Du,\nabla d^{2} \right\rangle \right)  ,
\]
and
\begin{align*}
A^{ij} (-r^{-2}(d^{2})_{i;j}  &  +C\hat{u}_{i;j} )= -r^{-2} A^{ij}%
(d^{2})_{i;j}\\
&  + C\left(  nWH+\dfrac{f+W^{2}}{W^{2}}\left\langle \Pi_{\ast}\bar{\nabla
}_{D_{0}}D_{0},Du \right\rangle +A^{ij}\gamma_{ij} \right)  ,
\end{align*}
where
\[
A^{ij}(d^{2})_{i;j}=\Delta(d^{2}) -\dfrac{1}{W^{2}}\left\langle \nabla_{Du}
\nabla d^{2}, Du \right\rangle .
\]
Inserting the previous expressions into \eqref{eqint2}, we have
\begin{align*}
&  \dfrac{C^{2} f}{W^{2}}\left(  |Du|^{2}-\dfrac{2}{Cr^{2}}\left\langle
Du,\nabla d^{2} \right\rangle \right)  g^{\prime\prime}\\
& \ \ \ \ \ + \left[  C\left(  nWH+\dfrac{f+W^{2}}{W^{2}}\left\langle \Pi_{\ast}%
\bar{\nabla}_{D_{0}}D_{0},Du \right\rangle +A^{ij}\gamma_{ij} \right) \right. \\ 
& \ \ \ \ \ \ \ \hspace{3cm} -\left.
r^{-2} \left(  \Delta(d^{2}) -\dfrac{1}{W^{2}}\left\langle \nabla_{Du} \nabla
d^{2}, Du \right\rangle \right)  \right]  g^{\prime}\\
&  \ \ \ \ \ \ \ \ \ \ \ \ \hspace{5cm} \leq Mg +A^{i} ( -r^{2} (d^{2}%
)_{i}+C\hat{u}_{i})g^{\prime}.
\end{align*}
Using the explicit expression of $A^{i}$, it is easy to see that $C A^{i}%
\hat{u}_{i}$ contains bounded terms and the term
\[
C\left(  nWH+\dfrac{f+W^{2}}{W^{2}}\left\langle \Pi_{\ast}\bar{\nabla}_{D_{0}%
}D_{0},Du \right\rangle \right)  .
\]
Therefore, we conclude that
\[
\dfrac{C^{2} f}{W^{2}}(|Du|^{2}-\dfrac{2}{Cr^{2}}\left\langle Du,\nabla d^{2}
\right\rangle ) g^{\prime\prime} +Pg^{\prime}-Mg\leq0,
\]
where $P$ and $M$ do not depend on $u$. Finally, it is easy to check that the
coefficient of $g^{\prime\prime}$ is strictly positive if we assume that
$|Du|\geq\dfrac{16 u_{0}}{r}$. It implies that
\[
W(p)\leq C_{2}=\sup_{B_{r} (o)} f+\dfrac{16 u_{0}}{r}.
\]
Since $p$ is the maximum point of $h$, this implies that
\[
(e^{\frac{C_{1}}{2}}-1)W(0)\leq C_{2} e^{C_{1}}.
\]

\end{proof}


For the proof of Theorem \ref{main} we make use of the following lemma, which shows that the SC condition implies an explicit mean convexity condition. Precisely:

\begin{lemma}
\label{rmksc}Assume $M$ is a
Hadamard manifold satisfying the strict convexity condition and such that $K_{M}\leq-\alpha$, for some constant $\alpha>0$. Then $M$ satisfies the $h$-mean convexity
condition for $h<\sqrt{\alpha}$, that is,  given $x\in\partial_{\infty}M,$
 a relatively open subset $W\subset\partial_{\infty}M$ containing $x$ and $h<\sqrt{\alpha}$,
there exists a $C^{2}$ open set $\Lambda\subset\overline{M}$ such that
$x\in\operatorname*{Int}(\partial_{\infty}\Lambda)\subset W$ and the  mean curvature of $M\backslash\Lambda$ with respect to the normal vector pointing to $M\backslash\Lambda$ is bigger than or equal to $h$.
\end{lemma}
\begin{proof}
Given $x\in\partial_{\infty}M$ and
 a relatively open subset $W\subset\partial_{\infty}M$ containing $x,$ let $\Omega$ be a convex unbounded domain in $M$, given by the SC condition such that $x\in\operatorname*{Int}(\partial_{\infty}\Omega)\subset W$.  Denote by
$d:\Omega\rightarrow \mathbb{R}$ the distance function to $\partial\Omega$. Then the
 hessian comparison theorem (see \cite{choi}) yields 
\[
\Delta d\geq(n-1)\sqrt{\alpha}\tanh(\sqrt{\alpha}d),
\]
i.e. the equidistant hypersurface $\Omega_{d}$ of $\Omega$ is $\sqrt{\alpha}\tanh
(\sqrt{\alpha}d)$-convex. Since $\tanh(\sqrt{\alpha}d)\underset{d\rightarrow \infty}{\longrightarrow}1$, we
deduce that $M$ also satisfies the $h$-mean convexity
condition for $h<\sqrt{\alpha}$.
\end{proof}

\begin{proof}
[Proof of Theorem \ref{main}]Let $\gamma:\left(  -\infty,\infty\right)
\rightarrow M$ be the geodesic translated by $Y.$ Set $P=\exp_{o}\left\{
Y(o)\right\}  ^{\bot}$ where $o=\gamma\left(  0\right)  .$ 
Let $p\in P$ and $t\in\mathbb{R}$ be given. We may write $p=\exp_{\gamma(s)}u$
for some $s\in\mathbb{R}$ and $u\in\gamma^{\prime}\left(  s\right)  ^{\bot}.$
Since $\tilde{\gamma}(r)=\exp_{\gamma(s)}\left(  ru\right)  ,$ $r\in\left[
0,1\right]  ,$ is a geodesic and $\varphi_{t}$ an isometry, $\beta
(r):=\varphi_{t}\left(  \tilde{\gamma}(r)\right)  $ is also a geodesic which,
moreover, satisfies the initial conditions
\begin{align*}
\beta(0) &  =\varphi_{t}(\tilde{\gamma}(0))=\varphi_{t}(\gamma(s))=\gamma\left(
s+t\right)  \\
\beta^{\prime}(0) &  =d\left(  \varphi_{t}\right)  _{\gamma(s)}u=:v,
\end{align*}
we have $\beta(r)=\exp_{\gamma(s+t)}(rv)$ by uniqueness. It follows that
\[
\varphi_{t}(p)=\beta(1)=\exp_{\gamma(s+t)}v.
\]
Moreover, since
\[
0=\left\langle u,\gamma^{\prime}(s)\right\rangle =\left\langle d\left(
\varphi_{t}\right)  _{\gamma(s)}u,d\left(  \varphi_{t}\right)  _{\gamma
(s)}\gamma^{\prime}(s)\right\rangle =\left\langle v,\gamma^{\prime
}(s+t)\right\rangle
\]
we have $v\in\gamma^{\prime}\left(  s+t\right)  ^{\bot}$ and, as the normal
exponential map of a geodesic in Hadamard manifold is a diffeomorphism from
the normal bundle of the geodesic onto $M,$ we have $\varphi_{t}\left(
p\right)  \cap P\neq\varnothing$ if and only if $t=0.$ 

We now observe that $Y$ is everywhere transversal to $P$. Indeed, assume by contradiction that $Y$ is not transversal to $P$ at some point $p\in P.$ Let
$d:N\rightarrow\mathbb{R}$ be the distance function to $P.$ We set  $f(t)=d\left(  \varphi\left(  t,p\right)  \right)  $ and observe that $f(0)=0$. Moreover, since
$\varphi(t,\cdot)$ is an isometry of $N,$ we have, for any fixed $t,$
$$
\operatorname{grad}d\left(  \varphi\left(  t,p\right)  \right)  =d\varphi\left(
t,p\right)  _{p}\left(  \operatorname{grad}d\left(  p\right)  \right).
$$
Therefore, we obtain
\begin{align*}
f^{\prime}(t)  & =\left\langle \operatorname{grad}d,\left.  \frac{\partial
\varphi\left(  s,p\right)  }{\partial s}\right\vert _{s=t}\right\rangle
=\left\langle \operatorname{grad}d\left(  \varphi\left(  t,p\right)  \right)
,Y\left(  \varphi\left(  t,p\right)  \right)  \right\rangle \\
& =\left\langle d\varphi\left(  t,p\right)  _{p}\left(  \operatorname{grad}%
d\left(  p\right)  \right)  ,d\varphi\left(  t,p\right)  _{p}\left(  Y\left(
p\right)  \right)  \right\rangle \\
& =\left\langle \operatorname{grad}d\left(
p\right)  ,Y\left(  p\right)  \right\rangle =0.
\end{align*}
This implies that $f\equiv0$ and, in return, that $\varphi\left(  t,p\right)  \in P$ for all $t,$ which yields to a contradiction.
This proves that $P$ is a Killing section.

 Since any orbit of $\varphi$ at $\partial_{\infty}N$ intersects $\Gamma$ at one
and exactly one point, $\Gamma$ is the Killing graph of a function $\phi\in
C^{0}\left(  \partial_{\infty}P\right)  .$ Let
$F\in C^{2,\alpha}(P)\cap C^{0}(\bar{P})$ $(\bar{P}=P\cup\partial_{\infty}P)$
be such that $F|_{\partial_{\infty}P}=\phi$.

Let $\rho$ be the geodesic distance in $P$ to a fixed point $o\in P$. We
denote by $B_{k}$, for $k=2,3,\ldots$, the geodesic ball in $P$ centered in $o$ and of radius
$k  $. We first show that, for any
$k=2,3,\ldots$, there is a solution $u_{k}\in C^{2,\alpha}(\bar{B}_{k})$ of
\begin{equation}
\left\{
\begin{array}
[c]{ll}%
Q_H[u_{k}]=0, & \text{on}\ B_{k}\\
u_{k}|_{\partial B_{k}}=F_{k}=F|_{\partial B_{k}}. &
\end{array}
\right.  \label{solk}%
\end{equation}
In order to prove the existence of the $u_k$'s, we will need some a priori height estimate. More precisely, we claim that given some $k\geq 2$, there is a constant $C_{j}$ depending
only on $j$ such that if $u_{k}$ is a solution of (\ref{solk}) and $j\leq k$
then $\sup_{B_{j}}\left\vert u_{k}\right\vert \leq C_{j}$ $.$ 
Let us prove the claim. We choose two open subsets $U_{\pm}$ of
$\gamma\left(  \pm\infty\right)  $ in $\partial_{\infty}M$. Using the SC condition, we obtain the existence of two $C^{2}$
convex subsets $W_{\pm}$ of $M$ such that $\partial_{\infty}W_{\pm}\subset
U_{K_{\pm}}.$ Denote by $K_{\pm}$ the hypersurfaces $K_{\pm}=\partial W_{\pm
}.$ As observed in Lemma \ref{rmksc} and since $|H|<\sqrt{\alpha}$, we may assume that $K_{\pm}$ are $H_{0}$
mean convex with $H_{0}\geq H.$ We then choose $C_{j}$ such that the orbit of
$\left\{  \varphi_{t}\right\}  $ through a point of $B_{j}$ intersects
$W_{\pm}$ for some $t\geq C_{j}$. It is clear that we may assume that the
Killing graph of $F$ does not intersect $K_{\pm}.$ The claim then follows from the tangency principle.

We have two important consequences of the previous height estimate. The first one is that problem (\ref{solk})
is solvable for any $k\geq 2.$ In fact, the only missing hypothesis to apply
Theorem 1 of \cite{DL} to guarantee the solvability of (\ref{solk}) are on the
Ricci curvature of $M$ and on the mean curvature of the Killing cylinder over the boundary of $B_k$ (see \cite{DL}, Theorem 1)$.$ Concerning the hypothesis on the mean curvature of the Killing cylinder over the boundary of $B_k$, we claim that it holds true in our setting. Indeed, since the orbits of $\varphi_{t}$ are equidistant curves of $\gamma$, it
follows that the Killing cylinder $K_{k}$ over $\partial B_{k}$ is an
equidistant hypersurface of $\gamma.$ Therefore the mean curvature $H_{K_{k}}$
of $K_{k}$ with respect to the inner normal vector field of $K_{k}$ coincides
with the Laplacian of the distance to $\gamma.$ One may then apply the hessian
comparison theorem to obtain%
\[
H_{K_{k}}\geq   \sqrt{\alpha}\tanh\left(  k\sqrt{\alpha
}\right)  \geq H.
\]

 Moreover, a direct
inspection on the proof of Theorem 1 of \cite{DL} shows that the hypothesis on the Ricci curvature is only
used to obtain a priori height estimates, which we just obtained above.

Secondly, the a priori height estimates we obtained above, Theorem
\ref{intgrad} and classical Schauder estimate for linear elliptic PDE
(see \cite{GT}) guarantee the compactness of the sequence of solutions $\left\{
u_{k}\right\}  $ on compact subsets of $M.$ Then, by the diagonal method, the
sequence $\left\{  u_{k}\right\}  $ contains a subsequence converging
uniformally in $C^{2}$ norm on compact subsets of $M$ to a global solution
$u\in C^{\infty}\left(  P\right)  $ of $Q_H[u]=0,$ where $|H|<\sqrt{\alpha}.$ It remains to show that $u$
extends continuously to $\partial_{\infty}P$ and that $u|_{\partial_{\infty}%
P}=\phi$. 

Let $(x_{k})_{k}$ be a sequence of points of $P$ converging to $x\in
\partial_{\infty}P$. Since $\overline{P}$ is compact, there exists a
subsequence $\varphi(u(x_{k_{j}}),x_{k_{j}})$ of $\varphi(u(x_{k}),x_{k})$ which
converges to $z\in\overline{P}$. Since $x_{k}$ diverges and $\varphi(u(x_{k_{j}%
}),x_{k_{j}})\in Gr(u)$, we have that $z\in\partial_{\infty}Gr(u)$. We claim
that $z\in Gr(\phi)$. To prove the claim, we will show that if $z\in
\partial_{\infty}M\backslash Gr(\phi)$ then $z\notin\partial_{\infty}Gr(u)$.
Let $z\in\partial_{\infty}M\backslash Gr(\phi)$. Since $Gr(\phi)$ is compact
and $z\notin Gr(\phi)$, using the SC condition, we
can find an hypersurface $E\subset M$ such that $\partial_{\infty}E$ separates
$z$ and $Gr(\phi)$. Moreover, using Lemma \ref{rmksc}, the mean curvature of $E$ with respect to the
unit normal vector field pointing to the connected component $U$ of
$M\backslash E$ whose asymptotic boundary contains $Gr(\phi)$, is larger or equal to
$h$ for $h< \sqrt{\alpha}$. Since $u_{k}|_{\partial B_{k}}\underset{k\rightarrow\infty
}{\longrightarrow}\phi$, there exists $k_{0}$ such that, for all $k\geq k_{0}%
$, $\partial Gr(u_{k})\subset U$ and $\partial Gr(u_{k})\cap E=\varnothing$.
By the tangency principle and using that $|H|<\sqrt{\alpha}$, we deduce that, for all $k\geq k_{0}$,
$Gr(u_{k})\subset U$. It follows that $z\notin\partial_{\infty}Gr(u)$. This
proves the claim i.e. $z\in Gr(\phi)$. In particular, it follows that $u$ is bounded.

Now, since $\partial_{\infty}Gr(u)\subset Gr(\phi)$, there exists $x_{0}%
\in\partial_{\infty}P$ such that $z=\varphi(u(x_{0}),x_{0})$. Using that $u$ is
bounded, we deduce there exists a subsequence $\{u(x_{k_{j_{i}}})\}$ which
converges to some $t_{0}\in\mathbb{R}$. It follows from the fact that the extension of
$\varphi_{t_{0}}$ to $\overline{P}$ is continuous that
\[
z=\lim_{i\rightarrow\infty}\varphi(u(x_{k_{j_{i}}}),x_{k_{j_{i}}})=\varphi(t_{0},x).
\]
Since $\varphi:\mathbb{R}\times\partial_{\infty}P\rightarrow\partial_{\infty}M$ is
injective, we deduce that $t_{0}=\phi(x_{0})$ and $x_{0}=x$. Since this last
fact holds true for every converging subsequences, we have proved that
$u(x_{k})\underset{k\rightarrow\infty}{\longrightarrow}\phi(x)$. This
concludes the proof of Theorem \ref{main}.
\end{proof}
\bigskip

\bibliographystyle{acm}
\bibliography{referenceskillingsub1}

\end{document}